\documentclass{article}

\usepackage{amsthm,amssymb,amsmath,amsfonts,dsfont}
\usepackage{esint}

\theoremstyle{plain}
\newtheorem{thm}{Theorem}[section]
\newtheorem{lemma}[thm]{Lemma}
\newtheorem{cor}[thm]{Corollary}

\theoremstyle{definition}
\newtheorem{defn}[thm]{Definition}
\newtheorem{rmk}[thm]{Remark}
\newtheorem{ex}[thm]{Example}

\usepackage{url}

\makeatletter
\newcommand{\address}[1]{\gdef\@address{#1}}
\newcommand{\email}[1]{\gdef\@email{\url{#1}}}
\newcommand{\emaill}[1]{\gdef\@emaill{\url{#1}}}
\newcommand{\@endstuff}{\par\vspace{\baselineskip}\noindent\small
\begin{tabular}{@{}l}\scshape\@address\\\textit{E-mail address:} \@email\\\textit{E-mail address:} \@emaill\end{tabular}}
\AtEndDocument{\@endstuff}
\makeatother

\address{Department of Mathematics, The University of Texas at Austin,\\ \scshape \hspace*{-.5cm}2515 Speedway, Stop C1200, Austin TX 78712-1202, USA}
\email{caffarel@math.utexas.edu}
\emaill{tomasetti@math.utexas.edu}

\title{Fully Nonlinear Equations with Applications to Grad Equations in Plasma Physics}
\author{Luis A. Caffarelli and Ignacio Tomasetti\footnote{Authors are supported by the NSF Grant DMS 1500871.}}

\begin{document}

\maketitle

\begin{abstract} In this paper we generalize an equation studied by Mossino and Temam in \cite{mossinotemam}, to the fully nonlinear case. This equation arises in plasma physics as an approximation to Grad equations, which were introduced by Harold Grad in \cite{grad}, to model the behavior of plasma confined in a toroidal vessel. We prove existence of a $W^{2,p}$-viscosity solution and regularity up to $C^{1,\alpha}(\overline{\Omega})$ for any $\alpha<1$(we improve this regularity near the boundary). The difficulty of this problem lays on a right hand side which involves the measure of the superlevel sets, making the problem nonlocal.
\end{abstract}

\section*{Introduction}

We will consider $W^{2,p}$-viscosity solutions $u:\Omega\subset\mathds{R}^n\longrightarrow\mathds{R}$ for
\begin{equation}\label{pde} 
  \begin{cases}
    F(D^2u(x))=g\Big(\vert u\geq u(x)\vert\Big) & \text{in $\Omega$}\\
    u=\psi & \text{on $\partial\Omega$}
  \end{cases}
\end{equation}
where $\Omega\subset\mathds{R}^n$ is an open, bounded and connected set, with $C^{1,1}$ boundary. The operator $F:\mathcal{S}\longrightarrow\mathds{R}$, is a convex, uniformly elliptic operator with ellipticity constants $0<\lambda\leq\Lambda$, where $\mathcal{S}:=\{$real n$\times$n symmetric matrices$\}$. For simplicity we will assume $F(0)=0$. We will also require that $F$ satisfies the following structure condition
\begin{equation}\label{structurecondition}
\mathcal{M}^-(M-N)\leq F(M)-F(N)\leq\mathcal{M}^+(M-N),
\end{equation}
for all $M,N\in\mathcal{S}$. Here, $\mathcal{M}^-$ and $\mathcal{M}^+$ are the extremal Pucci operators
$$\mathcal{M}^-(M)=\lambda\sum_{e_i>0}e_i+\Lambda\sum_{e_i<0}e_i,$$
and
$$\mathcal{M}^+(M)=\Lambda\sum_{e_i>0}e_i+\lambda\sum_{e_i<0}e_i,$$
where $e_i=e_i(M)$ are the eigenvalues of $M$. In the right hand side of \eqref{pde}, $\vert\cdot\vert$ denotes the n-dimensional Lebesgue measure, and $g:[0;\vert\Omega\vert]\longrightarrow\mathds{R}$ is a continuous function. We will adopt the notation 
$$u\geq u(x):=\{y\in \Omega:u(y)\geq u(x)\}$$
for the superlevel sets of $u$. Finally, we consider a boundary value $\psi\in W^{2,p}(\Omega)$ for $p>n$.
\bigskip

The motivation to study this problem is to generalize Grad equations in plasma physics, and its approximations, to nonlinear operators. These equations were introduced by Harold Grad in \cite{grad}, and appear in the literature as Queer Differential Equations(QDE), or Grad Equations. They arise in modelling plasma, which is confined under magnetic forces in a toroidal container. Grad noticed that a simplified version of plasma equations was possible using $u^*$, the increasing rearrangement of $u$:
$$u^{*}(t):=\inf\{s:\vert u<s\vert\geq t\}.$$
Here is where we start building a connection with our approximation problem (\ref{pde}). Notice that heuristically, $u^*$ is the inverse of the measure of the sublevel sets of $u$. In \cite{grad}, he demonstrated that there are profile functions $\mu$ and $\nu$ which are prescribed by the dynamics of the plasma; consequently his equation reads
$$\Delta u=-\mu'(u)(u^*{'})^\gamma-\gamma\mu(u)(u^{*}{'})^{\gamma-2}u^{*}{''}-\frac{1}{2}(\nu^2(u)){'}(u^{*}{'})^2-\nu^2(u)u^{*}{''}$$
for some power $\gamma$. For clarity we avoided the arguments: $u$ and its derivatives are evaluated at some point $x$ while the rearrangements and its derivatives are evaluated at $t:=\vert u<u(x)\vert$. Many authors attacked the problem trying to approximate these equations. The first one was introduced by Roger Temam in \cite{temam}, and then improved by Mossino and Temam in \cite{mossinotemam}. They studied properties of directional derivatives of the rearrangement function, and proved existence results for
$$\Delta u(x)=g(\vert u<u(x)\vert,u(x))+f(x).$$
Years later, Laurence and Stredulinsky, in \cite{lsnewapproach} and \cite{lsbootstrap}, studied a model equation, closer to Grad's formulation. They considered the particular case when $\gamma=2$, $\mu\equiv 1/2$ and $\nu\equiv 0$ obtaining
$$\Delta u(x)=-u^{*}{''}(\vert u<u(x)\vert).$$
Even this simplified case presents many difficulties. The authors introduced a very interesting approach to the problem: they described an approximation with solutions to a $N-$free boundary problem. In order to apply this process they assumed extra regularity for the level sets of a solution, which is mentioned later in Section 3.
\bigskip

The idea behind this paper is the following: all of these previous papers addressed the problem with a variational method for the Laplacian; instead, we will use a viscosity approach for a general family of fully nonlinear operators. A similar equation to the one of Mossino and Temam is studied, and \textbf{even for the case with the Laplacian we improve the regularity results}. 
\bigskip

The paper is organized as follows: In the first section we cite some preliminary definitions. Mainly, we state the basics of $W^{2,p}$-viscosity solutions. The classic viscosity solutions' theory does not apply to this particular problem because of our right hand side in (\ref{pde}). Disregarding the regularity of $g$ and $u$, we notice that having $\vert u=c\vert>0$ for some constant $c$ makes the right hand side of our equation discontinuous. Therefore, we adopt this $W^{2,p}$-viscosity notion defined in \cite{ccks} by Caffarelli, Crandall, Kocan and Swiech  which allows merely measurable ``ingredients''. In their paper they proved existence and interior $W^{2,p}$-estimates for solutions to an equation with a fixed right hand side $f(x)$. Strongly based on their results, Winter in \cite{winter} extended this regularity up to the boundary proving global $W^{2,p}$-estimates for viscosity solutions and an existence result for $W^{2,p}$-strong solutions. For clarity in the presentation, the results from the literature that will be used through the paper will be addressed at the Appendix A.

In section 2 we state and prove the main theorem of existence and global regularity. The idea of the proof is to
\begin{itemize}
\item freeze $u$ in the right hand side
\item solve the resulting equation using \cite{ccks} theory
\item build a sequence of right hand sides and solutions
\item use a fixed point argument and a convergence theorem to find a solution
\end{itemize}

In section 3 we prove more regularity under additional hypothesis. As long as $\vert\nabla u\vert$ is uniformly bounded below, or equivalently, if we have a uniform interior ball condition for the level sets of $u$, then we have $C^{0,\alpha}$ regularity for the right hand side. This estimate turns into $C^{2,\alpha}$ regularity for the solution $u$. We cannot ensure regularity for the level sets, but if we start with a regular enough domain, say $\partial\Omega$ with a uniform interior ball condition, then we gain $C^{2,\alpha}$ regularity for $u$ in a neighborhood of the boundary.

\section{Preliminary Definitions}

First we are going to present the definitions of viscosity solutions for fully nonlinear equations with measurable ingredients, described in the paper of Caffarelli-Crandall-Kocan-Swiech \cite{ccks}. In this setting we work with the problem
\begin{equation}\label{pdegeneral} 
  \begin{cases}
    F(D^2u(x))=f(x) & \text{in $\Omega$}\\
    u=\psi & \text{on $\partial\Omega$}
  \end{cases}
\end{equation}
where our right hand side is a fixed measurable function $f$.

\begin{defn}
Let be $F$ a uniformly elliptic operator, $f\in L^{p}(\Omega)$ for $p>n/2$. Let $u:\Omega\longrightarrow\mathds{R}$ be a continuous function, we say it is a \textbf{$W^{2,p}$-viscosity subsolution} of \eqref{pdegeneral} in $\Omega$, if $u\leq \psi$ on $\partial\Omega$ and the following holds: for all $\varphi\in W^{2,p}(\Omega)$ such that $u-\varphi$ has a local maximum at $x_0\in \Omega$ then 
$$ess\limsup_{x\rightarrow x_0}F(D^2\varphi(x))-f(x)\geq 0.$$
We define supersolutions in the same way; $u$ is a \textbf{$W^{2,p}$-viscosity supersolution} of \eqref{pdegeneral} in $\Omega$, if $u\geq \psi$ on $\partial\Omega$ and the following holds: for all $\varphi\in W^{2,p}(\Omega)$ such that $u-\varphi$ has a local minimum at $x_0\in \Omega$ then 
$$ess\liminf_{x\rightarrow x_0}F(D^2\varphi(x))-f(x)\leq 0.$$
\end{defn}

\begin{rmk}
We can also use this alternative definition for $W^{2,p}$-viscosity subsolutions. For all $\varphi\in W_{loc}^{2,p}(\Omega)$, for all $\varepsilon>0$, and $O\subset\Omega$ open such that
$$F(D^2\varphi(x))-f(x)\leq -\varepsilon,$$
a.e. in $O$, then $u-\varphi$ cannot have a local maximum in $O$.
\end{rmk}

Because we will use Winter's results, we also add the definition of $W^{2,p}$-strong subsolutions. 
\begin{defn}
In the same setting as before, $u$ is a \textbf{$W^{2,p}$-strong subsolution} of \eqref{pdegeneral} in $\Omega$, if $u\leq \psi$ on $\partial\Omega$ and
$$F(D^2u(x))\geq f(x)$$
a.e. in $\Omega$.
\end{defn}

\section{Main result}

In this section we state and prove existence and a first global regularity result.
\begin{thm}\label{main}
Our problem (\ref{pde})
\begin{equation*}
  \begin{cases}
    F(D^2u(x))=g\Big(\vert u\geq u(x)\vert\Big) & \text{in $\Omega$}\\
    u=\psi & \text{on $\partial\Omega$}
  \end{cases}
\end{equation*}
with the setting given in the introduction, has a $W^{2,p}$-viscosity solution $u$. Furthermore, $u\in W^{2,p}(\Omega)$ and we have the following estimate
$$\Vert u\Vert_{W^{2,p}(\Omega)}\leq C\Big[\Vert u\Vert_{L^{\infty}(\Omega)}+\Vert \psi\Vert_{W^{2,p}(\Omega)}+\Vert g(\vert u\geq u(x)\vert)\Vert_{L^{p}(\Omega)}\Big].$$
\end{thm}

\begin{cor}
Using Sobolev embedding theorem we get that a solution is in $C^{1,\alpha}(\overline{\Omega})$ for any $\alpha<1$, provided that $\psi\in W^{2,p}$ for every $p>n$.
\end{cor}
\bigskip

The structure of the proof for Theorem \ref{main} is somehow simple; we set an approximating problem \eqref{pdeepsilon}, we prove the existence of a solution for it and then we take the limit to obtain the solution to \eqref{pde}. Before presenting this approximating problem in Lemma \ref{approxlemma}, we give a quick explanation on the reasoning behind it. Recall that the results from the appendix will be used next: existence and uniqueness, fixed point, and convergence.

If in \eqref{pde} we freeze a function $v\in Lip(\Omega)$ for the right hand side, i.e. $f_v(x):=g\Big(\vert \{y\in \Omega:v(y)\geq v(x)\}\vert\Big)$, we get
\begin{equation}\label{pdefixedrighthandside}
  \begin{cases}
    F(D^2u(x))=f_v(x) & \text{in $\Omega$}\\
    u=\psi & \text{on $\partial\Omega$}.
  \end{cases}
\end{equation}
Then the hypothesis of Theorem \ref{winter} are satisfied and there exists a unique $W^{2,p}$-viscosity solution $u$ to \eqref{pdefixedrighthandside}. The next step would be to apply the fixed point Theorem \ref{fixed} for the application $T(v)=u$. The problem is that we cannot ensure continuity for $T$ because of the right hand side of \eqref{pdefixedrighthandside}. Not even if we require more regularity for $v$(not even $C^\infty$ works). We will overcome this inconvenient solving an auxiliary problem with a smoothened right hand side which allows us to perform the fixed point argument. Given $v\in Lip(\Omega)$, $\varepsilon>0$, consider
\begin{equation}\label{pdefixedrighthandsideepsilon}
  \begin{cases}
    F(D^2u(x))=f_v^\varepsilon(x):=g\Big(\dfrac{1}{\varepsilon}\int_0^\varepsilon\vert v\geq v(x)-h\vert dh\Big) & \text{in $\Omega$}\\
    u=\psi & \text{on $\partial\Omega$}.
  \end{cases}
\end{equation}
Because $f_v^\varepsilon\in L^p(\Omega)$, using Theorem \ref{winter} we have existence and uniqueness of a $W^{2,p}$-viscosity solution $u\in W^{2,p}(\Omega)$ to \eqref{pdefixedrighthandsideepsilon} with the estimate
\begin{equation*}
  \Vert u\Vert_{W^{2,p}(\Omega)}\leq C\Big[\Vert u\Vert_{L^{\infty}(\Omega)}+\Vert \psi\Vert_{W^{2,p}(\Omega)}+\Vert f_v^\varepsilon\Vert_{L^{p}(\Omega)}\Big].
\end{equation*}
Now we can state our approximation lemma.

\begin{lemma}\label{approxlemma}
Given $\varepsilon>0$, there exists a $W^{2,p}$-viscosity solution $u_\varepsilon$ to
\begin{equation}\label{pdeepsilon}
  \begin{cases}
    F(D^2u(x))=f_u^\varepsilon(x), & \text{in $\Omega$}.\\
    u=\psi, & \text{on $\partial\Omega$}.
  \end{cases}
\end{equation}
\end{lemma}

\begin{proof}
The existence is proved, as we remarked, using the fixed point Theorem \ref{fixed}. We define $T:Lip(\Omega)\longrightarrow Lip(\Omega)$ as the application defined by \eqref{pdefixedrighthandsideepsilon} and the existence and uniqueness theorem, i.e., $T(v)=u$. In order to prove the hypothesis required for $T$, we will make use of the convergence Theorem \ref{convergence}.

\textit{Continuity of T:} If we consider $v_k\overset{Lip}{\longrightarrow}v$, then, does $u_k:=T(v_k)\overset{Lip}{\longrightarrow}T(v)$?
We know that $u_k\in W^{2,p}(\Omega)$ and 
\begin{equation*}
  \Vert u_k\Vert_{W^{2,p}(\Omega)}\leq C\Big[\Vert u_k\Vert_{L^{\infty}(\Omega)}+\Vert \psi\Vert_{W^{2,p}(\Omega)}+\Vert f_v^\varepsilon\Vert_{L^{p}(\Omega)}\Big]\leq \widetilde{C}
\end{equation*}
with $\widetilde{C}$ independent on $k$. This is achieved using Alexandroff-Bakelman-Pucci (ABP) estimates
 \begin{equation*}
  \sup_{\Omega}u_k\leq \sup_{\partial\Omega}u_k+C\Vert f_v^\varepsilon\Vert_{L^{p}(\Omega)}
\end{equation*}
and the equivalent for the $\inf_{\Omega}u_k$. This ABP version for measurable ingredients is stated in Caffarelli-Crandall-Kocan-Swiech(Proposition 3.3 in \cite{ccks}). We also have the estimate
 \begin{equation*}
  \Vert f_v^\varepsilon\Vert_{L^{p}(\Omega)}\leq \vert\Omega\vert^{1/p}g\big(\vert\Omega\vert\big)
\end{equation*}
which makes $\widetilde{C}$ even independent on $\varepsilon$.
Now consider $u_{k_j}$ any subsequece of $u_k$. Using Rellich-Kondrachov theorem we can find a subsequence $u_{k_{j_i}}$ of $u_{k_j}$(for simplicity we will use the notation $u_i:=u_{k_{j_i}}$) converging to some $u_\varepsilon$ in the Lipschitz norm, i.e., $u_i\overset{Lip}{\longrightarrow}u_\varepsilon$. If we can prove that $u_\varepsilon$ is the unique $W^{2,p}$-viscosity solution to \eqref{pdefixedrighthandsideepsilon}($T(v)=u_\varepsilon$), then we have the convergence $u_k=T(v_k)\overset{Lip}{\longrightarrow}u_\varepsilon=T(v)$. Therefore, we obtain the continuity for $T$.

Every $u_i\in C^0(\Omega)$ is the unique $W^{2,p}$-viscosity solution to \eqref{pdefixedrighthandsideepsilon} with $v_i$ in the right hand side. We have $\Omega_i=\Omega$, and $F_i=F$ fixed for every $i$. The convergence $u_i\overset{Lip}{\longrightarrow}u_\varepsilon$ implies the locally uniformly convergence. So we only need to check the convergence 
$$\Vert f^\varepsilon_v(x)-f^\varepsilon_{v_i}(x)]\Vert_{L^{p}(B_r(x_0))}\longrightarrow 0$$
in order to satisfy all the hypothesis in the convergence Theorem \ref{convergence}. We know that $v_i\overset{Lip}{\longrightarrow}v$, then $\delta_i:=\Vert v_i-v\Vert_{L^{\infty}}\longrightarrow 0$. Thus let $x$ and $h$ fixed,
$$\vert v_i\geq v_i(x)-h\vert\leq\vert v+\delta_i\geq v(x)-\delta_i-h\vert\searrow\vert v\geq v(x)-h\vert$$ 
as $i\longrightarrow\infty$ and also
$$\vert v_i\geq v_i(x)-h\vert\geq\vert v-\delta_i\geq v(x)+\delta_i-h\vert\nearrow\vert v>v(x)-h\vert.$$ 
We can show that $\vert v>v(x)-h\vert=\vert v\geq v(x)-h\vert$ for a.e. $h$ in $[0;\varepsilon]$. This happens if and only if $\vert v=v(x)-h\vert=\vert v^{-1}(v(x)-h)\vert=0$ for a.e. $h\in[0;\varepsilon]$. A corollary of Rademacher theorem, says that if $v$ is a Lipschitz function then for a.e. $y\in v^{-1}(\alpha)$, $\nabla v(y)=0$. Therefore $\vert v^{-1}(v(x)-h)\vert=\vert v^{-1}(v(x)-h)\cap\{\nabla v(y)=0\}\vert$. Using a corollary of the Coarea formula we get also that $\mathcal{H}^{n-1}\Big( v^{-1}(v(x)-h)\cap\{\nabla v(y)=0\}\Big)=0$. Here $\mathcal{H}^{n-1}$ stands for the $(n-1)$-dimensional Hausdorff measure. Then for every $x\in\Omega$ we get the convergence 
$$\vert v_i\geq v_i(x)-h\vert\longrightarrow\vert v\geq v(x)-h\vert$$ 
for a.e. $h$. Applying the dominated convergence theorem first, and the continuity of $g$ we get
\begin{equation}\label{gconvergence}
g\Big(\dfrac{1}{\varepsilon}\int_0^\varepsilon\vert v_i\geq v_i(x)-h\vert dh\Big)\longrightarrow g\Big(\dfrac{1}{\varepsilon}\int_0^\varepsilon\vert v\geq v(x)-h\vert dh\Big)
\end{equation}
as $i\longrightarrow\infty$. This last result is the pointwise convergence of $f^\varepsilon_{v_i}$ to $f^\varepsilon_{v}$. Again, applying the dominated convergence theorem we get the $L^p$ convergence needed. So all the hypothesis are satisfied to apply the theorem and therefore $T$ is continuous.
 
\textit{Compactness of T:} Let $v_k$ a bounded sequence in $Lip(\Omega)$ then $u_k:=T(v_k)\in W^{2,p}(\Omega)$ is bounded as before. After Rellich-Kondrachov there exists a convergent subsequence.

\textit{Boundedness of the eigenvectors:} We have to prove that the set
$$\Gamma:=\{v\in Lip(\Omega): \exists\gamma\in[0;1] \text{ such that }v=\gamma T(v)\}$$
is bounded. Suppose by contradiction that it is not. First we note that $0\in\Gamma$ with 
$\gamma=0$, and for every $0\neq v\in\Gamma$ the $\gamma$ associated with $v$ is not zero. Suppose then that there exist a sequence of nonzero elements $v_k\in\Gamma$, and a respective sequence $\gamma_k$ such that $v_k=\gamma_kT(v_k)$ and $\Vert v_k\Vert_{Lip(\Omega)}\longrightarrow\infty$. Because $v_k\in Lip(\Omega)$, then $v_k/\gamma_k\in W^{2,p}(\Omega)$ and
$$\Vert v_k\Vert_{Lip(\Omega)}\leq\Vert \dfrac{v_k}{\gamma_k}\Vert_{Lip(\Omega)}\leq C\Vert \dfrac{v_k}{\gamma_k}\Vert_{W^{2,p}(\Omega)}\leq \widetilde{C}$$
which is a contradiction. Therefore $\Gamma$ is bounded.

The hypothesis of Schaefer's theorem are satisfied, so there exists a Lipschitz fixed point $u_\varepsilon$ for $T$, i.e., $u_\varepsilon=T(u_\varepsilon)$. Moreover, by Theorem \ref{winter}, $u_\varepsilon$ is a $W^{2,p}$-viscosity solution to \eqref{pdeepsilon}, which is in $W^{2,p}(\Omega)$.
\end{proof}

The purpose of finding such a $u_\varepsilon$, was to approximate a solution for \eqref{pde}. Then the following question is if we can take the limit $\varepsilon\longrightarrow\infty$. 

\begin{proof}[Proof of Theorem \ref{main}]
For every $\varepsilon>0$ we have a solution $u_\varepsilon\in W^{2,p}(\Omega)$ with uniformly bounded $W^{2,p}$ norm(with respect to $\varepsilon$). Then there exists a subsequence(that we will also call $u_\varepsilon$) and a Lipschitz function $u$ such that $u_\varepsilon\overset{Lip}{\longrightarrow} u$. So $u_\varepsilon\longrightarrow u$ locally uniformly and we will be able to apply again the convergence Theorem \ref{convergence}. In this case we have on the right hand sides, the $L^p$ functions
$$f_u(x):=g\big(\vert u\geq u(x)\vert\big)$$
and
$$f_{u_\varepsilon}^\varepsilon(x):=g\Big(\dfrac{1}{\varepsilon}\int_0^\varepsilon\vert u_\varepsilon\geq u_\varepsilon(x)-h\vert dh\Big).$$
We are left to prove the convergence
$$\Vert f_u-f^\varepsilon_{u_\varepsilon}(x)\Vert_{L^{p}(B_r(x_0))}\longrightarrow 0.$$
By triangle inequality
$$\Vert f_u-f^\varepsilon_{u_\varepsilon}(x)\Vert_{L^{p}(\Omega)}\leq\Vert f_u-f^\varepsilon_{u}(x)\Vert_{L^{p}(\Omega)}+\Vert f_u^\varepsilon-f^\varepsilon_{u_\varepsilon}(x)\Vert_{L^{p}(\Omega)}.$$
We have that the second term goes to zero as in previous calculations \eqref{gconvergence}. We will use a similar argument for bounding the first term.
\begin{align*}
\vert u\geq u(x)\vert&\leq \vert u\geq u(x)-h\vert\\
&=\vert u\geq u(x)\vert+\vert u(x)>u\geq u(x)-h\vert\\
&\leq\vert u\geq u(x)\vert+\vert u(x)>u\geq u(x)-\varepsilon\vert.
\end{align*}
Then
\begin{align*}
\dfrac{1}{\varepsilon}\int_0^\varepsilon\vert u\geq u(x)-h\vert dh&\geq\dfrac{1}{\varepsilon}\int_0^\varepsilon\vert u\geq u(x)\vert dh\\
&=\vert u\geq u(x)\vert
\end{align*}
and
\begin{align*}
\dfrac{1}{\varepsilon}\int_0^\varepsilon\vert u\geq u(x)-h\vert dh&\leq
 \dfrac{1}{\varepsilon}\int_0^\varepsilon\vert u\geq u(x)\vert+\vert u(x)>u\geq u(x)-\varepsilon\vert dh\\
 &=\vert u\geq u(x)\vert+\vert u(x)>u\geq u(x)-\varepsilon\vert.
\end{align*}
Therefore 
\begin{equation*}
\dfrac{1}{\varepsilon}\int_0^\varepsilon\vert u\geq u(x)-h\vert dh\searrow \vert u\geq u(x)\vert
\end{equation*}
as $\varepsilon\longrightarrow\infty$. Accordingly we obtained pointwise convergence for $f^\varepsilon_u$ to $f_u$, which after the dominated convergence theorem implies the convergence on the $L^p$ norm.

Hypothesis of Theorem \ref{convergence} are satisfied and we finally obtain our main result: $u$ a $W^{2,p}$-viscosity solution to \eqref{pde}, in $W^{2,p}(\Omega)$ and  with the corresponding estimates.
\end{proof}

The last remark of this section is that we obtain an explicit formula for the $0$ Dirichlet problem in a ball. 
\begin{ex}\label{ball}
When $\Omega=B_r(x_0)$, $F=\Delta$, $g(t)=-t$ and $\psi=0$ we have the solution:
\begin{equation*}
\tilde{u}(x)=\frac{\omega_n}{2n(n+2)}\Big[r^{n+2}-\vert x-x_0\vert^{n+2}\Big]
\end{equation*}
where $\omega_n$ is the measure of the n-dimensional unit ball. In a similar way we can prove that $\frac{1}{\Lambda}\tilde{u}$ is a solution when $F=\mathcal{M}^-$(respectively $\frac{1}{\lambda}\tilde{u}$ for $F=\mathcal{M}^+$). We will use this example in the next section to build subsolutions that can be used as barriers to prove gradient bounds.
\end{ex}

\section{Further Regularity}

In order to gain more regularity for our solution $u$ we probably need to get some regularity for the right hand side $f_u$. So far, in the case when $u$ has flat regions, $f_u$ is not even continuous. In principle, this discontinuity does not depend on the regularity of $u$ but on its flat regions. We can prove that under the negativity of $g$, $u$ is not allowed to have these flat regions with positive measure.

\begin{rmk}
Let be $u$ a solution for (\ref{pde}) with right hand side $g<0$ in $(0;\vert\Omega\vert]$, then 
$$\vert u=a\vert = 0$$
for every constant $a\in\mathbb{R}$.
\end{rmk}

\begin{proof}
Suppose that there exists an $a\in\mathds{R}$ such that $A:=\{u=a\}$ satisfies that $\vert A\vert>0$. Then, by a classic result from Stampacchia we obtain that
\begin{equation*}
\nabla u(x)=0 \quad\text{for a.e. $x\in A$}.
\end{equation*}
Now we can define $A':=A\cap\{\nabla u=0\}$, and apply Stampacchia's result again
\begin{equation}\label{laplacianocero}
D^2u(x)=0 \quad\text{for a.e. $x\in A'$}.
\end{equation}
We are left with the set $A'':=A'\cap\{D^2u=0\}$ with the same measure $\vert A''\vert=\vert A'\vert=\vert A\vert>0$. By the definition of $u$ being a $W^{2,p}$-strong supersolution of \eqref{pde}, we have that for a.e $x$ in $\Omega$
$$F(D^2 u(x))-g\Big(\vert u\geq u(x)\vert\Big)\leq 0.$$
Moreover, for $x_0\in A''$, the argument inside $g$ is strictly positive; $\vert u\geq u(x_0)\vert\geq\vert A''\vert>0$. So in the particular case when $g<0$ in $(0;\vert\Omega\vert]$ we get the contradiction
$$F(D^2 u(x_0))-g\Big(\vert u\geq u(x_0)\vert\Big)> 0.$$
\end{proof}

Then, for this specific case we obtain continuity for $f_u$. But we need at least $C^{0,\alpha}$ regularity on $f_u$ in order to apply Schauder type estimates to obtain $u\in C^{2,\alpha}$. We will have this regularity in two particular cases listed in the next two theorems. The first one is an adaptation(simplification) to Laurence and Stredulinsky's theorem and requires an additional lower bound for the gradient.

\begin{thm}[Theorem 3.1 in \cite{lsbootstrap}]\label{gradientbound}
Let $u\in W^{2,p}_0(\Omega)$ with a uniform lower bound $\vert\nabla u\vert>c_0>0$ in the set $\Omega_{t_0}:=\{y\in \Omega:u(y)< t_0\}$, where $t_0<\Vert u \Vert_{L^\infty}$ and $c_0=c_0(t_0)$. Then $f_u\in C^1(\Omega_{t_0})$.
\end{thm}
In other words, the theorem asserts that if we have an uniform lower bound for the gradient(away from the maximum of $u$), then we get: regularity for the level sets of $u$ and we discard a possible ``flatness'' which ruins the smoothness of $f_u$. The proof presented in \cite{lsbootstrap} includes an approximation argument by $C^\infty_0$ functions and coarea formula. 

This last theorem translates into regularity for our problem. We state this in the following corollary.
\begin{cor}
If we have a solution $u$ to our problem (\ref{pde}), with $0$ boundary condition and a gradient lower bound as in Theorem \ref{gradientbound}, then $f_u\in C^1(\Omega_{t_0})$, and therefore $u\in C^{2,\alpha}(\Omega_{t_0})$.
\end{cor}

\begin{proof}
In order to get $C^{2,\alpha}$ estimates we just need to apply the classical theory of viscosity solutions for fully nonlinear equations as in Chapter 8 from \cite{cc}. Recall that at this point we have a right hand side in $C^1(\Omega_{t_0})$ which allows us to use classical viscosity solutions instead of $W^{2,p}$-viscosity solutions.
\end{proof}

The second theorem states that, under certain conditions, a barrier argument implies lower bounds as in Theorem \ref{gradientbound}.

\begin{thm}
If $\Omega$ has a uniform inner ball condition (i.e., for any point $y$ in $\partial\Omega$, there exists a ball $B_\varepsilon\subset\Omega$ with $\varepsilon>\varepsilon_0>0$ and $y\in\partial B_\varepsilon$), then $\vert\nabla u\vert>c_0>0$ in a neighborhood of $\partial\Omega$, where $c_0=c_0(\Vert u\Vert_{C^{1,\alpha}(\overline{\Omega})})$. We consider the case where $g(t)=-t$ and $u=0$ on $\partial\Omega$.
\end{thm}
\begin{proof}
If we pick any point $y\in\partial\Omega$ we can touch it with a ball $B_{\varepsilon_0}\subset\Omega$. As in Example \ref{ball} we can build a an explicit solution $\tilde{u}$ in $B_{\varepsilon_0}$ for $F=\mathcal{M}^-$. Now we apply comparison between $u$ and $\tilde{u}$ in order to get gradient estimates. Without loss of generality we can take $\varepsilon_0$ small enough, such that 
$$\vert u\geq u(x)\vert\geq \frac{1}{2}\vert\Omega\vert\geq\vert B_{\varepsilon_0}\vert$$
for every $x\in B_{\varepsilon_0}$. This is possible because of the continuity of $u$ and of the right hand side, i.e.,
$$\vert u\geq t\vert \xrightarrow[t\longrightarrow 0]{}\vert\Omega\vert.$$
If this is the case then
\begin{align*}
\mathcal{M}^-(D^2 \tilde{u}(x))&=-\vert \{\tilde{u}\geq \tilde{u}(x)\}\cap B_{\varepsilon_0}\vert\\
&\geq -\frac{1}{2}\vert\Omega\vert\\
&\geq -\vert u\geq {u}(x)\vert\\
&=F(D^2u(x))\\
&\geq \mathcal{M}^-(D^2u(x))
\end{align*}
in $ B_{\varepsilon_0}$. In addition, we have that $0=\tilde{u}\leq u$ at $\partial B_{\varepsilon_0}$. So comparison applies and forces $\tilde{u}\leq u$ in $ B_{\varepsilon_0}$. Therefore we also have a lower bound for the gradient at the boundary, with the estimate
$$\vert\nabla u\vert\geq \tilde{u}_{-\nu}=\frac{\omega_n}{2n\Lambda}{\varepsilon_0}^{n+1}=c_0>0$$
where $-\nu$ is the inner normal to $\partial\Omega$. Finally we can extend a lower bound(say $c_0/2$) to a neighborhood of the boundary of $\Omega$ which will depend on the $C^{1,\alpha}(\overline{\Omega})$ norm of $u$. 
\end{proof}

\begin{rmk}
We can repeat this argument as long we have uniform inner ball conditions for the level sets $\{u=t\}$, and so, $C^{2,\alpha}$ regularity for the solution in that annulus.
\end{rmk}

\begin{rmk}
We expect this condition to be satisfied for convex domains, where we deduce the solutions will have convex level sets.  On the other hand, for nonconvex domains, in particular for dumbbell shaped domains, we expect to have a singular critical point where the superlevel sets separate into two components.
\end{rmk}

\appendix
\section{Appendix}

In this appendix section we gather the results from the literature that will be used in the proof of the main Theorem \ref{main}. First we have an existence and uniqueness result when the right hand side is fixed. We refer to Winter's version because it includes additional $W^{2,p}$ bounds for the unique solution.
\begin{thm}[Winter 4.6 in \cite{winter}]\label{winter}
Let $F$ be a convex operator satisfying the structure condition \eqref{structurecondition} and $F(0)=0$, $f\in L^{p}(\Omega)$ for $p>n$, $\psi\in W^{2,p}(\Omega)$ and $\partial\Omega\in C^{1,1}$. Then, there exists a unique $W^{2,p}$-viscosity solution to
\begin{equation*}
  \begin{cases}
    F(D^2u(x))=f(x), & \text{in $\Omega$}.\\
    u=\psi, & \text{on $\partial\Omega$}.
  \end{cases}
\end{equation*}
Moreover, $u\in W^{2,p}(\Omega)$ and
\begin{equation*}
  \Vert u\Vert_{W^{2,p}(\Omega)}\leq C\Big[\Vert u\Vert_{L^{\infty}(\Omega)}+\Vert \psi\Vert_{W^{2,p}(\Omega)}+\Vert f\Vert_{L^{p}(\Omega)}\Big]
\end{equation*}
for $C=C(n,\lambda,\Lambda,p,\Omega).$
\end{thm}
Second, we introduce a classic fixed point theorem that will be crucial to extract a solution to our problem, out of a family of approximations.
\begin{thm}[Schaefer Fixed Point Theorem]\label{fixed}
Let $T:V\longrightarrow V$ a continuous and compact mapping, with $V$ a Banach space, such that the set
$$\{v\in V: \exists\gamma\in[0;1] \text{ such that }v=\gamma T(v)\}$$
is bounded. Then $T$ has a fixed point.
\end{thm}
Third, we will need the next powerful convergence result that will be used for proving continuity in the fixed point argument. And later for proving convergence of the solutions to auxiliary problems.
\begin{thm}[Caffarelli-Crandall-Kocan-Swiech 3.8 in \cite{ccks}]\label{convergence}
Let $\Omega_k\subset\Omega_{k+1}$ a sequence of subdomains of $\Omega$ converging to $\Omega$. Let $F$ and $F_k$ be uniformly elliptic operators with the same ellipticity constants and satisfying the structure condition \eqref{structurecondition}. Let $f\in L^p(\Omega)$ and $f_k\in L^p(\Omega_k)$ for $p>n$. Let $u_k\in C^0(\Omega_k)$ be $W^{2,p}$-viscosity subsolutions(supersolutions) of
$$F_k(D^2u(x))=f_k(x)$$
in $\Omega_{k+1}$, with $u_k$ converging locally uniformly to $u$ in $\Omega$. Finally assume that for every $B_r(x_0)\subset\Omega$, and for every $\varphi\in W^{2,p}(B_r(x_0))$ we have
$$\Vert[F_k(D^2\varphi(x))-f_k(x)-F(D^2\varphi(x))+f(x)]^+\Vert_{L^{p}(B_r(x_0))}\longrightarrow 0$$
$$\Big(\Vert[F_k(D^2\varphi(x))-f_k(x)-F(D^2\varphi(x))+f(x)]^-\Vert_{L^{p}(B_r(x_0))}\longrightarrow 0\Big),$$
Then $u$ is a $W^{2,p}$-viscosity subsolution(supersolution) of
$$F(D^2u(x))=f(x)$$
in $\Omega.$
\end{thm}
Finally, we note that 
\begin{rmk}
Due to a result by Escauriaza in \cite{escauriaza}, we can extend $p$ to the case where $p>n-\varepsilon_0$ for some universal $\varepsilon_0$ in Theorems \ref{winter} and \ref{convergence}.
\end{rmk}

\bibliographystyle{plain}
\bibliography{bib}

\end{document}